\documentclass{amsart}
\usepackage{amsmath,amssymb}
\usepackage{graphicx,tikz}
\newtheorem{theorem}{Theorem}[section]
\newtheorem{proposition}[theorem]{Proposition}
\newtheorem{corollary}[theorem]{Corollary}

\newtheorem{lemma}[theorem]{Lemma}

\theoremstyle{definition}

\theoremstyle{remark}

\numberwithin{equation}{section}

\newcommand{\al}{\alpha}
\newcommand{\be}{\beta}
\newcommand{\de}{\delta}
\newcommand{\ep}{\varepsilon}

\newcommand{\ga}{\gamma}
\newcommand{\ka}{\kappa}
\newcommand{\la}{\lambda}

\newcommand{\si}{\sigma}

\newcommand{\vp}{\varphi}

\newcommand{\De}{\Delta}

\newcommand{\La}{\Lambda}
\newcommand{\Si}{\Sigma}
\newcommand{\Om}{\Omega}

\def\RR{\mathbb{R}}

\renewcommand\SS{\mathbb{S}}

\newcommand\bh{{\bar h}}

\newcommand{\cH}{{\mathcal H}}

\newcommand{\cW}{{\mathcal W}}

\newcommand{\pd}{\partial}
\newcommand\minus\backslash
\newcommand{\id}{{\rm id}}

\newcommand\lan\langle
\newcommand\ran\rangle

\newcommand{\Span}{\operatorname{span}}

\DeclareMathOperator\dist{dist}

\newcommand\uno{1}
\renewcommand\leq\leqslant
\renewcommand\geq\geqslant
\newlength{\intwidth}

\begin{document}

\title[Nodal sets in each conformal class]{Prescribing the nodal set of the first eigenfunction in each
  conformal class}

\author{Alberto Enciso}
\address{Instituto de Ciencias Matem\'aticas, Consejo Superior de
  Investigaciones Cient\'\i ficas, 28049 Madrid, Spain}
\email{aenciso@icmat.es}

\author{Daniel Peralta-Salas}
\address{Instituto de Ciencias Matem\'aticas, Consejo Superior de
  Investigaciones Cient\'\i ficas, 28049 Madrid, Spain}
\email{dperalta@icmat.es}

\author{Stefan Steinerberger}
\address{Department of Mathematics, Yale University, New Haven, CT 06511, USA}
\email{stefan.steinerberger@yale.edu}

%
%
\begin{abstract}
We consider the problem of prescribing the nodal set
  of the first nontrivial eigenfunction of the Laplacian in a conformal class. Our main result is that, given
  a separating closed hypersurface $\Si$ in a compact Riemannian manifold
  $(M,g_0)$ of dimension $d\geq 3$, there is a metric $g$ on $M$ conformally equivalent to $g_0$ and with the same volume such that the nodal set of
  its first nontrivial eigenfunction is a $C^0$-small deformation of
  $\Si$ (i.e., $\Phi(\Si)$ with $\Phi:M \rightarrow M$ a
  diffeomorphism arbitrarily close to the identity in the $C^0$~norm). 
\end{abstract}
\maketitle

\section{Introduction}
\label{S.intro}

Let $M$ denote a closed manifold of dimension
$d\geq3$ endowed with a fixed Riemannian metric $g_0$. The
eigenfunctions of $M$ satisfy the equation
\[
\De u_k=-\la_k u_k\,,
\]
where $0=\la_0<\la_1\leq\la_2\leq\dots$
are the eigenvalues of $M$ and $\De$ is the Laplace operator of the
manifold. The zero set $u_k^{-1}(0)$ is called the {\em nodal set}\/
of the eigenfunction, and each connected component of $M\backslash
u_k^{-1}(0)$ is known as a {\em nodal domain}\/.

The study of the eigenvalues and eigenfunctions of a manifold is a
classic topic in geometric analysis with a number of important open
problems~\cite{Ya82,Ya93}. A fundamental fact is that the behavior of
$\la_k$ as $k\to\infty$ is extremely rigid, as captured by Weyl's
law. Major open questions roughly related to this rigidity phenomenon
concern the asymptotic behavior as $k\to\infty$ of, say, the number of
nodal domains and the measure of the nodal set of $u_k$~\cite{DF,HS} or the number of critical points of the eigenfunctions~\cite{JN}.

In striking contrast, the low-energy behavior of the eigenvalues is
quite flexible. A landmark in this direction is the proof that one
can prescribe an arbitrarily high number of eigenvalues of the
Laplacian, including multiplicities. More precisely~\cite{Colin2},
given any finite sequence of positive real numbers $\la_1\leq
\la_2\leq \dots\leq \la_N$, there is a metric~$g$ on $M$ having this
sequence as its first $N$ nontrivial eigenvalues. This cannot be
accomplished if we require the metric~$g$ to be conformally equivalent
to the original metric $g_0$ and of the same volume; in particular, it
is known~\cite{Soufi,Friedlander,LiYau} that the supremum of the first nontrivial
eigenvalue $\la_1$ corresponding to~$g$ is finite as~$g$ ranges over
the set of metrics conformal to $g_0$ and with the same volume. 

The nodal sets of the low-energy eigenfunctions turn out to be
surprisingly flexible too. Indeed, it has been recently shown that~\cite{EP14}, given
a separating hypersurface $\Si$ in $M$, there is a metric~$g$ on the
manifold for which the nodal set $u_1^{-1}(0)$ of the first
eigenfunction is precisely~$\Si$. Throughout, we will assume that the hypersurfaces are all smooth, connected, compact and without boundary. We recall that a
hypersurface $\Si$ is {\em separating}\/ if its complement $M\minus
\Si$ is the union of two disjoint open sets. Similar results for higher
eigenfunctions have been established as well. 
The goal of this paper is to show that the nodal set $u_1^{-1}(0)$ of
the first eigenfunction can be prescribed, up to a small deformation, even if we require the
metric~$g$ to be conformal to the original metric $g_0$ and of the
same volume. More precisely, we will prove the following theorem.

\begin{theorem}\label{T.main}
  Let $M$ be a $d$-manifold ($d\geq3$) endowed with a Riemannian
  metric $g_0$ and let $\Si$ be a separating hypersurface. Then, given any $\de>0$,
  there is a metric $g$ in $M$ conformally equivalent to $g_0$ and
  with the same volume such that its first eigenvalue $\la_1$ is simple and the nodal set of its first eigenfunction $u_1$ is $\Phi(\Si)$, where $\Phi$ is a diffeomorphism of
  $M$ whose distance to the identity in $C^0(M)$ is at most $\de$.  
\end{theorem}

The proof of this result is based on the explicit construction of a conformal
factor which is of order~$\ep$ in a neighborhood $\Om_\eta$ of
$\Sigma$ of width $\eta$. Geometrically, this ensures that the
manifold, endowed with the rescaled metric, has the structure of a
dumbbell, as depicted in Figure~1.  The basic idea behind the proof of
the theorem is to exploit this dumbbell structure through a fine
analysis of the first eigenfunction. More precisely, we show that as $\ep$ tends to zero the first
eigenfunction approximates a harmonic function in
$\Om_\eta$ with constant boundary values. In turn, the zero set of
this harmonic function can then be controlled provided that $\eta$ is small. Related results on level sets with prescribed
topologies were derived, using completely different methods, for
Green's functions in~\cite{JDG}, for harmonic functions in $\RR^n$
in~\cite{Adv} and for eigenfunctions of the Laplacian in~\cite{EP14}. 

\begin{figure}[t]
\begin{center}
\begin{tikzpicture}[scale = 0.85]
\draw [thick] (-7,0) to [out=270,in=180] (-5,-2)  to [out=0,in=180] (-3,-2) to [out=0, in=270] (-2,0) 
to [out=90, in = 0] (-3,2) to [out=180, in = 0] (-3,2) to [out=180, in = 0] (-5,2) to [out=180, in = 90] (-7,0);

\draw [ultra thick] (-4,2) to [out=250,in=110] (-4,-2);
\draw [dashed] (-4.4,2) to [out=250,in=110] (-4.4,-2);
\draw [dashed] (-3.6,2) to [out=250,in=110] (-3.6,-2);
\node (A) at (-4.1,0.9) {$\Sigma$};

\draw [thick] (0,0) to [out=270,in=180] (2,-2) to [out=0,in=180] (3,-0.5) to [out=0,in=180] (4,-2) to [out=0, in=270] (5,0) 
to [out=90, in = 0] (4,2) to [out=180, in = 0] (4,2) to [out=180, in = 0] (3,0.5) to [out=180, in = 0] (2,2) to [out=180, in = 90] (0,0);

\draw [ultra thick] (3,0.5) to [out=250,in=110] (3,-0.5);
\draw [dashed] (3.3,0.6) to [out=250,in=110] (3.3,-0.6);
\draw [dashed] (2.7,0.6) to [out=250,in=110] (2.7,-0.6);
\node (A) at (-4.1,0.9) {$\Sigma$};

\end{tikzpicture}
\end{center}
\caption{A separating hypersurface $\Sigma$ and the $\eta$-neighbourhood $\Om_\eta$. Making the metric small in $\Om_{\eta}$ turns $M$ into a dumbbell.}
\end{figure}
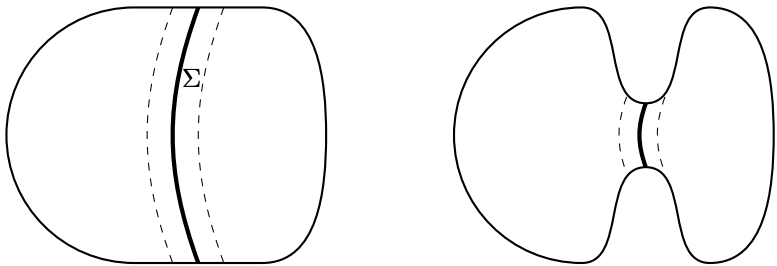

An easy application of Theorem~\ref{T.main} enables us to prove that,
given any Riemannian manifold $(M,g_0)$, there is a metric conformally equivalent to $g_0$ and with the same volume such that the
eigenfunction $u_1$ has as many isolated critical points as one
wishes.

\begin{theorem}\label{T.cp}
  Let $M$ be a $d$-manifold endowed with a Riemannian metric $g_0$,
  with $d\geq3$, and let $N$ be a positive integer. Then there is a
  metric $g$ on $M$, conformally equivalent to $g_0$ and with the same
  volume, such that its first nontrivial eigenfunction $u_1$ has at least~$N$
  non-degenerate critical points.
\end{theorem}

It is worth recalling that, on surfaces, Cheng~\cite{Cheng} gave a
topological bound for the number of critical points of the
$k^{\mathrm{th}}$ eigenfunction that lie on the nodal line. We do not know if results analogous to Theorems~\ref{T.main} and~\ref{T.cp} hold for surfaces.
The proof of Theorem~\ref{T.main} is given in Section~\ref{S.proof}, although the proofs of several technical lemmas are relegated to
Sections~\ref{S:L1}-\ref{S:coro}. The proof of Theorem~\ref{T.cp}, which
hinges on Theorem~\ref{T.main}, is then presented in Section~\ref{S.cp}.

\section{Proof of the main theorem}
\label{S.proof}

We divide the proof of Theorem~\ref{T.main} in five steps. In Step~1
we define a discontinuous metric $g_\ep$ that is conformal to $g_0$
and is of order~$\ep$ in a neighborhood $\Om_\eta$ of $\Sigma$ of
width $\eta$; here we have Lemma~\ref{L} showing
that the first nontrivial eigenvalue of $(M,g_\ep)$ is simple and
tends to zero as $\ep\to0$. In Step~2 we exploit the partial
regularity of the metric $g_\ep$ to obtain estimates in certain mixed
Sobolev norms (cf. Lemma~\ref{L.uniform}), which imply that the first
eigenfunction $u_\ep$ is H\"older continuous and that its $L^\infty$
norm is bounded uniformly in~$\ep$. These estimates are used in Step~3 to show that $u_\ep$ converges in $\Om_\eta$ to a harmonic function $h$ with constant boundary values on $\partial\Om_\eta$ (cf.\ Proposition~\ref{P.Ck}). In Step~4 the nodal set of $h$ is analyzed, the main result being Corollary~\ref{C.levelset} where we prove that it is a regular level set diffeomorphic to $\Sigma$ provided that the width $\eta$ is small. Finally, the proof of the theorem is completed in Step~5 taking the two independent parameters $\ep$ and $\eta$ sufficiently small, and using that $0$ is a regular value of $h$. 

\subsubsection*{Step 1: Defining a discontinuous metric}
For small enough $\eta>0$, the set
\[
\Om_\eta:=\big\{ x\in M: \dist_{g_0}(x,\Si)<\eta\big\}
\]
is diffeomorphic to $(-\eta,\eta)\times\Si$. Denoting by
$\Om_\eta^c:=M\backslash{\Om_\eta}$ the complement of the set $\Om_\eta$, let us consider the bounded
discontinuous function
\[
f_\ep(x):=\begin{cases}
\ep& \text{if }x\in\Om_\eta\,,\\
\ka & \text{if }x\in\Om_\eta^c\,.
\end{cases}
\]
Here the constant $\ka\equiv \ka(\ep,\eta)$ is defined as
\[
\ka:=\bigg(1+(1-\ep^{\frac d2})\frac{ |\Om_\eta|}{|\Om_\eta^c|}\bigg)^{2/d}\,,
\]
with $|\cdot|$ standing for the volume of a set computed with respect
to the metric $g_0$. With this choice of $\ka$, the volume of $M$ with
the
metric $g_\ep:=f_\ep g_0$ is independent of the choice of~$\ep$ and~$\eta$. Notice that $\ka$ is well-behaved since, for
$\ep, \eta$ sufficiently small,
\[
1\leq \ka\leq \bigg(1+ \frac{ |\Om_\eta|}{|\Om_\eta^c|}\bigg)^{2/d}\, \leq 2.
\]

The first (nontrivial) eigenfunction of $M$ with the discontinuous
metric $g_\ep$ is a minimizer of the Rayleigh quotient
\begin{align*}
q_\ep(u)&:=\frac{\int_M |du|_\ep^2\, dV_\ep}{\int_M u^2\, dV_\ep}\\[1mm]
&=\frac{\ep^{\frac
    d2-1}\int_{\Om_\eta} |du|^2+\ka^{\frac
    d2-1}\int_{\Om_\eta^c}|du|^2} {\ep^{\frac
    d2}\int_{\Om_\eta} u^2+\ka^{\frac
    d2}\int_{\Om_\eta^c} u^2}
\end{align*}
in the space of nonzero functions $u\in H^1(M)$ such that
\[
0=\int_Mu\, dV_\ep=\ep^{\frac
    d2}\int_{\Om_\eta} u+\ka^{\frac
    d2}\int_{\Om_\eta^c} u\,.
\]
Throughout we are denoting with a subscript $\ep$ the quantities
(norm, Riemannian measure) associated with the metric $g_\ep$ and we are omitting
the measure under the integral sign when it is the one
corresponding to~$g_0$. We will use the notation $\la_\ep\equiv
\la_{1,\ep}$ for the first nonzero eigenvalue of $(M,g_\ep)$, and call
$u_\ep\equiv u_{1,\ep}$ its corresponding eigenfunction, which we assume to be normalized to have unit norm:
\[
\int_M u_\ep^2 \, dV_\ep=\ep^{\frac
    d2}\int_{\Om_\eta} u_\ep^2+\ka^{\frac
    d2}\int_{\Om_\eta^c} u_\ep^2=1\,.
\]

The small-$\ep$ behavior of the first nontrivial eigenvalue $\la_\ep$ is
described in the following lemma, here in particular it is shown to be simple. The proof is presented in Section~\ref{S:L1}.
With a little more work, and using that $M\backslash \Om_\eta$ consists
of precisely two connected components because $\Si$ separates, the argument in fact implies that, as $\ep \rightarrow 0$, the second
eigenvalue tends to the smallest of the first nonzero Neumann eigenvalue
of each connected component. 

\begin{lemma}\label{L}
For any small but fixed $\eta$, the first nontrivial
eigenvalue $\la_\ep$ is simple and converges to zero as
$\ep\searrow0$ (more precisely, $\la_\ep\leq C\ep^{\frac d2-1}$).
\end{lemma}

\subsubsection*{Step 2: Anisotropic derivative estimates for the eigenfunction}

In Step~1 we defined a metric~$g_\ep$ that is
discontinuous across the boundary of the set $\Om_\eta$. Two redeeming
features of this metric are that it is smooth everywhere but on
$\pd\Om_\eta$ and that $g_\ep$ is in fact everywhere smooth along the directions that
are tangent to $\pd\Om_\eta$. 

In view of the discontinuity of the metric, a simple, convenient way
of exploiting this partial regularity is by considering
mixed Sobolev norms. In order to define them, it is convenient to use
the following approach. Denote by $U_0$ a neighborhood of $\Om_\eta$ where the signed distance $\rho$ to the hypersurface $\Si$ as measured with respect the metric~$g_0$
is a smooth function. Set
\[
Y:=\chi_0\, \nabla \rho
\]
where $\chi_0\equiv \bar\chi_0(\rho)$ is a fixed nonnegative smooth
function that only depends on $\rho$, is equal to 1 in a
neighborhood $U_1\subset U_0$ of $\Om_\eta$ and
vanishes outside $U_0$. With this definition, $Y$ defines a smooth
vector field in~$M$ that does not vanish in $U_1$. 

Let us now take a finite collection of vector fields $X_1,\dots, X_s$
that are everywhere orthogonal to $Y$ (i.e., $g_0(X_j,Y)=0$),
supported in $U_0$ and such that 
\[
\Span \{ Y|_x, X_1|_x,\dots,  X_s|_x\}= T_xM \qquad \text{for all
} x\in U_1\,.
\]
A simple transversality argument shows that in fact one can take $s=d+1$. Given a set $V\subset M$,
functions in $H^1(V)$ that additionally have $k$
weak derivatives in the directions tangent to $\Si$ that are also in
$H^1$ will be characterized through the mixed Sobolev norm
\[
\|u\|_{H^1H^k_{\mathrm T}(V)}:=\sum_{j_1+\cdots+ j_s\leq k} \|X_1^{j_1}\cdots
X_s^{j_s}u\|_{H^1(V)}+\|(1-\chi_0)\, u\|_{H^{k+1}(V)}\,.
\]
We will write $\|\cdot\|_{H^1H^k_{\mathrm T}} \equiv \|\cdot\|_{H^1H^k_{\mathrm T}(M)}$. The following lemma, which is proved in Section~\ref{S:L2}, establishes an upper bound for the mixed Sobolev norm of $u_\ep$.

\begin{lemma}\label{L.uniform}
For any nonnegative integer $k$, there are constants
$C_\ep$ and~$C$ depending on $k$ such that the normalized eigenfunction~$u_\ep$ satisfies 
\[
\|u_\ep\|_{H^1H^k_{\mathrm T}}< C_\ep\,,\qquad
\|u_\ep\|_{H^1H^k_{\mathrm T}(\Om_\eta^c)}< C\,. 
\]
Here the constant $C$ is independent of~$\ep$.
\end{lemma}

Since $u_\ep$ is then in the Sobolev space $H^1H^k_{\mathrm T}$ for
any positive $k$, $u_\ep$ is H\"older continuous with exponent~$\al$
for any $\alpha<1/2$ on account of the Sobolev embedding theorem~\cite[Theorem 10.2]{Be78}. Further, from the second
inequality in the lemma we infer that the boundary traces
$$
B_\ep^{\pm}:=u_\ep|_{\partial\Om_\eta^{\pm}}
$$
are bounded in $H^k(\pd\Om_\eta^\pm)$ uniformly in $\ep$, i.e.,
\begin{equation}\label{eq:B}
\|B_\ep^{\pm}\|_{H^k(\partial\Om_\eta^\pm)}<C\,.
\end{equation}
Here we are denoting by $\Om_\eta^\pm$ the connected components of the
set $\Om_\eta^c$ (which are known to be precisely two because $\Si$ is
a separating hypersurface). Notice, in particular, that $B_\ep^{\pm}$ are smooth functions. Furthermore, the eigenfunction $u_\ep$ is uniformly bounded in $L^\infty$, as established in the following argument.

\begin{corollary}\label{C:Linf}
There is a constant $C$ independent of $\ep$ such that $\|u_\ep\|_{L^\infty(M)}<C$.
\end{corollary} 
\begin{proof}
In the domains $\Om_\eta^\pm$ the metric $g_\ep=\ka g_0$ is smooth, so the eigenfunction $u_\ep$ is smooth as well and satisfies the boundary value problem
\begin{equation*}
\Delta u_\ep=-\kappa\la_\ep u_\ep\quad\text{in } \Om_\eta^\pm\,,\qquad u_\ep|_{\partial{\Om_\eta^\pm}}=B_\ep^\pm\,.
\end{equation*}
Here and in what follows $\De$ stands for the Laplacian corresponding
to the metric $g_0$. The standard a priori bounds following from the
maximum principle then imply that
$$
\|u_\ep\|_{L^\infty(\Om_\eta^\pm)}<\|B_\ep^\pm\|_{L^\infty(\partial\Om_\eta^\pm)}+C \ka\la_\ep\| u_\ep\|_{L^\infty(\Om_\eta^\pm)}\,,
$$
for an $\ep$-independent constant $C$. Since $\la_\ep\to 0$ as $\ep\to 0$ by Lemma~\ref{L} and $B_\ep^\pm$ is bounded by Eq.~\eqref{eq:B}, we conclude that
$$
\|u_\ep\|_{L^\infty(\Om_\eta^\pm)}<C\,.
$$ 

Similarly, in the domain $\Om_\eta$ the eigenfunction $u_\ep$ is
smooth and solves the boundary value problem
\begin{equation*}
\Delta u_\ep=-\ep\la_\ep u_\ep\quad\text{in } \Om_\eta\,,\qquad u_\ep|_{\partial{\Om_\eta^\pm}}=B_\ep^\pm\,.
\end{equation*}
The same argument as before proves the desired estimate in $\Om_\eta$, and the corollary follows.
\end{proof}

The following result is a corollary of the proof of
Lemma~\ref{L.uniform}, and will be instrumental in Step~5 to prove the
existence of a smooth metric $g$ with the properties stated in
Theorem~\ref{T.main}. For the ease of notation, throughout the proof we will denote by $u_n$
the first eigenfunction of the metric $g_n$ described in this
corollary. One should not confuse this for the $n^{\mathrm{th}}$
eigenvalue of a fixed metric; indeed, only first
eigenfunctions will be relevant in what follows.

\begin{corollary}\label{C.smooth}
One can take a sequence of smooth metrics $g_n$, conformal to $g_0$
and with the same volume, such that the first eigenfunction $u_n$ of
$g_n$ has multiplicity one and converges to $u_\ep$ in
  $C^0(M)$ and in $C^k(S)$ for any compact set $S\subset M\backslash\partial\Om_\eta$.
\end{corollary}
\begin{proof}
In the proof of Lemma~\ref{L.uniform} it is shown that there is a sequence of smooth metrics $g_n:=f_ng_0$, with $f_n$ converging to $f_\ep$ pointwise in $M$ and in $C^k(S)$ for any compact set $S\subset M\backslash\partial\Om_\eta$, such that
\begin{equation}\label{corseq}
\lim_{n\to\infty}|\la_n-\la_\ep|=0\,, \qquad \lim_{n\to\infty}\|u_n-u_\ep\|_{C^0(M)}=0\,.
\end{equation}
Here, $\la_n$ is the first nontrivial eigenvalue of the Laplacian with the metric $g_n$ (which is simple) and $u_n$ is the corresponding (normalized) eigenfunction. By elliptic regularity the $C^0$ convergence of $u_n$ to $u_\ep$ can be promoted to $C^k$ convergence in any compact set $S$ contained in $M\backslash\partial\Om_\eta$. To fix the volume of the metric $g_n$ we can rescale it with a constant factor $\ga_n$ such that $\lim_{n\to\infty}\ga_n=1$ and
$$
|M|=\int_M dV_{\ga_ng_n}=\int_M \ga_n^{\frac d2}dV_n
$$ 
for all $n$, where $dV_n$ is the volume element associated with the
metric $g_n$. The corollary then follows trivially because this
rescaling does have any effect on the convergence~\eqref{corseq}.
\end{proof}

\subsubsection*{Step 3: Approximation by a harmonic function in $\Om_\eta$}
We will show that, as $\varepsilon \rightarrow 0$, $u_{\varepsilon}$ tends to constants outside of $\Omega_{\eta}$ while, inside $\Omega_{\eta}$,
it can be approximated by certain harmonic function $h$.
Indeed, let us define $h$ as the function in $\Om_\eta$ given by the only solution to
the boundary value problem
\[
\De h=0\quad\text{in } \Om_\eta\,,\qquad h|_{\pd\Om_\eta^\pm}=c_\pm\,,
\]
where the constants $c_\pm$ have been chosen so that
\begin{align*}
c_+^2|\Om_\eta^+|+ c_-^2|\Om_\eta^-|&=\ka_0^{-\frac d2}\,,\\
c_+|\Om_\eta^+|+ c_-|\Om_\eta^-|&=0\,,
\end{align*}
with
$$
\ka_0:=\ka|_{\ep=0}=\Big(1+\frac{|\Om_\eta|}{|\Om^c_\eta|}\Big)^{\frac{2}{d}}\,.
$$
These equations simply ensure that if we extend $h$ by constants
outside of $\Omega_{\eta}$, then the corresponding extension has mean
value 0 (i.e., it is orthogonal to constants) and its $L^2$ norm is 1.
This determines $c_\pm$ up to a global sign, which can be fixed by
requiring that 
\[
\pm c_\pm>0\,.
\]

\begin{proposition}\label{P.Ck}
If $k$ is any integer, one
can choose the sign of $u_\ep$ so that it converges to the harmonic
function $h$ in $\Om_\eta$ and to constants in the complement of this set:
\[
\lim_{\ep\searrow0} \|u_\ep-h\|_{C^k(\Om_\eta)}=0\,,\qquad
\lim_{\ep\searrow0} \|u_\ep- c_\pm\|_{C^0(\Om_\eta^\pm)}=0\,.
\]
\end{proposition}

This result is proved in Section~\ref{S:P}. Observe that this
proposition shows that, for $\ep$ small enough, the nodal set of the
eigenfunction $u_\ep$ can be controlled using the harmonic function
$h$, a fact that will be exploited in the next step.

\subsubsection*{Step 4: Analysis of the harmonic function $h$}

The following proposition is proved in Section~\ref{S:prooflevel}. It
provides an asymptotic form of the harmonic function $h$ defined in
Step~4 as the width $\eta$ tends to 0, thereby allowing us to understand its nodal set for small $\eta$. It is stated in terms of local coordinates $(\rho,y)$, with $\rho$ the signed distance to the hypersurface $\Si$
with respect to the metric $g_0$, the sign being chosen so that
$\pd\Om_\eta^\pm=\rho^{-1}(\pm\eta)$, and $y=(y_1,\dots, y_{d-1})$ are
local coordinates in $\Si$.For details about the construction of the
coordinate system, see the proof of Lemma~\ref{L.uniform} in
Section~\ref{S:L2}. 

\begin{proposition}\label{P.harm}
Consider the function in $\Om_\eta$ given by
\[
\bh:=\frac{c_++c_-}2+\frac{c_+-c_-}{2\eta}\rho\,.
\]
For any nonnegative integers $j,k$, in the above local coordinates $(\rho,y)$ we then have
\[
\lim_{\eta\searrow0} \eta^j\big\|\pd_\rho^j D_y^k(h-\bh)\big\|_{L^\infty(\Om_\eta)}=0\,.
\] 
\end{proposition}

Here we have abused the notation to state the proposition in a
coordinate-dependent way, which is slightly more convenient for our purposes. An equivalent intrinsic statement, using
the vector fields $Y,X_1,\dots, X_s$ introduced in Step~2, is
\[
\lim_{\eta\searrow0} \eta^j\big\|Y^j X_1^{k_1}\cdots X_s^{k_s}(h-\bh)\big\|_{L^\infty(\Om_\eta)}=0\,.
\] 

An easy consequence of the proposition is
that the nodal set $h^{-1}(0)$ of the harmonic function $h$ is
diffeomorphic to $\Sigma$ (through a diffeomorphism of $M$ that is
close to the identity in the $C^0$ norm) provided that the width
$\eta$ is small enough. Moreover, $0$ is a regular value of $h$, so
that $h^{-1}(0)$ is robust under $C^1$-small perturbations of the
function $h$, a property that will be key in Step~5. The proof of this corollary is given in Section~\ref{S:coro}.

\begin{corollary}\label{C.levelset}
For any $\de>0$ there is some $\eta_0>0$ such that
if $\eta<\eta_0$ the nodal set
$h^{-1}(0)$ is given by $\Psi(\Si)$, where $\Psi:M\to M$ is a
diffeomorphism with $\|\Psi-\id\|_{C^0(M)}<\de$. Moreover, $0$ is a regular value of $h$, that is, $\nabla h(x)\neq 0$ at any point $x\in h^{-1}(0)$.
\end{corollary}

\subsubsection*{Step 5: The nodal set of the eigenfunction}

By Corollary~\ref{C.levelset}, the width $\eta$ can be chosen small
enough so that the zero set $h^{-1}(0)$ is regular and given by
$\Psi(\Si)$, where $\Psi$ is a
diffeomorphism of $M$ with $\|\Psi-\id\|_{C^0(M)}$ as small as one
wishes. Let us fix some $\eta$ for which this property holds.

In turn, Proposition~\ref{P.Ck} ensures that, for any fixed $\de>0$,
one can then choose a small $\ep$ so that the difference $u_\ep-h$ is
smaller than $\de$ in $C^k(\Om_\eta)$. For small enough~$\de$, Thom's
isotopy theorem~\cite[Section 20.2]{AR} then implies that the first
eigenfunction~$u_\ep$ has a regular connected component of the level
set~$u_\ep^{-1}(0)$ diffeomorphic to $\Si$, and that the corresponding
diffeomorphism can be chosen arbitrarily close to the identity in
$C^0(M)$. Moreover, this component, which is contained in $\Om_\eta$,
is the whole nodal set of $u_\ep$ because the second estimate of
Proposition~\ref{P.Ck} shows that $u_\ep$ does not vanish in
$\Om_\eta^\pm$ for small $\ep$ as
\[
\lim_{\ep\searrow 0}\|u_\ep-c_\pm\|_{C^0(\Om_\eta^\pm)}=0\,.
\]

By Corollary~\ref{C.smooth}, there is a sequence of metrics $g_n$,
conformal to $g_0$ and with the same volume, such that their first
eigenfunction $u_n$ converges to $u_\ep$ uniformly in $M$ and
additionally in $C^k(S)$ for any compact set $S\subset
M\backslash\partial\Om_\eta$, and the corresponding eigenvalue $\la_n$
has multiplicity one and converges to $\la_\ep$. Since $u_\ep^{-1}(0)\subset\Om_\eta$, the fact that $u_n\to
u_\ep$ uniformly in $M$ then ensures that, for any large enough $n$, $u_n^{-1}(0)\cap
\overline{\Om_\eta^{\pm}}=\emptyset$. Therefore, another application of Thom's isotopy
theorem yields that, $u_n^{-1}(0)$ is a regular level set given by $\Phi_n(\Si)$, with $\Phi_n$ a
diffeomorphism of $M$ with $\|\Phi_n-\id\|_{C^0(M)}$ arbitrarily
small. The theorem then follows by
taking $g_n$ as the metric in the statement, for any sufficiently large $n$.

\section{Proof of Lemma~\ref{L}}\label{S:L1}
Let us denote by $\rho$ the signed distance to the hypersurface $\Si$
with respect to the metric $g_0$, the sign being chosen so that $\pd\Om_\eta^\pm=\rho^{-1}(\pm\eta)$.
To prove that $\la_\ep$ tends to zero, let us take the Lipschitz function
\[
u:=\chi_{\Om_\eta^-} +[1-a(\eta+\rho)]\, \chi_{\Om_\eta^{\phantom{-}}}+(1-2a\eta)\, \chi_{\Om_\eta^+}\,,
\]
where the constant
\[
a:=\frac{\ka^{\frac d2} |\Om_\eta^c| +\ep^{\frac
    d2}|\Om_\eta|}{2\eta\ka^{\frac d2}|\Om_\eta^+| +\ep^{\frac d2}
  \int_{\Om_\eta} (\eta+\rho)}
\]
has been chosen so that
\[
\int_M u\, dV_\ep=0\,.
\]
Since $\rho\geq-\eta$ in $\Om_\eta$, it is clear that
\[
0<a<\frac C{\eta}
\]
for some absolute constant $C$.

Since 
\[
\int_M|du|_\ep^2\, dV_\ep=a^2\ep^{\frac d2-1}\int_{\Om_\eta}
|d\rho|^2=a^2\ep^{\frac d2-1}|\Om_\eta|\leq Ca^2\ep^{\frac d2-1}\eta\leq \frac{C\ep^{\frac d2-1}}\eta\,,
\]
where we have used that $|d\rho|^2=1$, and
\[
\int_Mu^2\, dV_\ep > \ka^{\frac d2}|\Om_\eta^-|> C\,,
\]
we can take $u$ as a test function to show that
\[
\la_\ep\leq q_\ep(u)\leq \frac{C\ep^{\frac d2-1}}\eta\,,
\]
which tends to zero as $\ep\searrow0$. 

Calling $\la_{2,\ep}$ the second nontrivial eigenvalue of $(M,g_\ep)$, we now prove that there
is some positive constant $\La$, independent of $\ep$, such that
\begin{equation}\label{lajep}
\la_{2,\ep}\geq \La\,.
\end{equation}
By the min-max principle, $\la_{2,\ep}$ is
\begin{equation}\label{minmax}
\la_{2,\ep}=\inf_{W\in\cW} \max_{\vp\in W\minus\{0\}} q_\ep(\vp)\,,
\end{equation}
where $\cW$ stands for the set of $3$-dimensional linear subspaces
of $H^1(M)$.

Following~\cite{Colin, EP14}, let us consider the direct sum decomposition 
\begin{equation*}\label{decomp}
H^1(M)=\cH_1\oplus \cH_2\,,
\end{equation*}
where
\begin{align*}
\cH_1&:=\big\{ \varphi\in H^1(M): \De \varphi|_{\Om_\eta}=0\big\}\,,\\
\cH_2&:=\big\{ \varphi\in H^1(M): \varphi|_{\Om_\eta}\in H^1_0(\Om_\eta)\,,\; \varphi|_{\Om_\eta^c}=0\big\}\,.
\end{align*}
Notice that these subspaces are orthogonal in the sense that
\begin{equation}\label{orthog}
\int_M g_0(d\varphi_1,d\varphi_2)=0
\end{equation}
for all $\varphi_1\in \cH_1$ and $\varphi_2\in\cH_2$. Accordingly, any function $\varphi\in H^1(M)$ can be written as
\[
\varphi= \varphi_1+\varphi_2
\]
in a unique way, with $\varphi_i\in \cH_i$. To prove that there is a positive lower bound for the eigenvalues $\la_{2,\ep}$, we can assume that $\la_{2,\ep}<C_0$ where $C_0$ is a constant independent of $\varepsilon$. Consider the set $\cW'$ of
$3$-dimensional subspaces of~$H^1(M)$ such that 
\[
q_\ep(\vp)\leq C_0+1
\]
for all nonzero $\vp\in \cW'$. Therefore, by Eq.~\eqref{minmax} we have 
\begin{equation}\label{minmax2}
\la_{2,\ep}=\inf_{W\in\cW'} \max_{\vp\in W\minus\{0\}} q_\ep(\vp)\,.
\end{equation}
The observation now is that if $\vp$ is in a subspace belonging
to $\cW'$, then
\begin{equation}\label{vp2}
\int \vp_2^2\, dV_\ep \leq C\ep\,\int \vp^2 \, dV_\ep\,,
\end{equation}
where $C$ is an $\ep$-independent constant. Indeed, 
\begin{align*}
C_0+1\geq q_\ep(\vp)&\geq
                      \frac{\ep^{\frac{d}{2}-1}\int_{\Om_{\eta}}|d\vp_2|^2}{\int
                      \vp^2\, dV_\ep}\\
&= \frac {1}\ep\frac{\int_{\Om_\eta}|d\vp_2|^2}{\int_{\Om_\eta}\vp_2^2}\frac{\|\vp_2\|_\ep^2}{\|\vp\|_\ep^2}\\
& \geq \frac{\la_{\Om_\eta}}\ep\frac{\|\vp_2\|_\ep^2}{\|\vp\|_\ep^2}\,,
\end{align*}
where the positive, $\ep$-independent constant $\la_{\Omega_\eta}$ is the first Dirichlet eigenvalue of $\Om_\eta$ with the metric $g_0$. This inequality implies the estimate~\eqref{vp2}. A straightforward computation also shows that~\eqref{vp2} implies
\begin{equation}\label{eq.est}
\|\vp\|_\ep^2\leq (1+C\ep^{\frac12})\|\vp_1\|_\ep^2\,.
\end{equation}
A second observation is that if $\vp$ is in a subspace belonging
to $\cW'$, since $\vp_1$ is harmonic in $\Om_\eta$, standard elliptic estimates, the trace inequality and Eq.~\eqref{eq.est} imply
\begin{equation}\label{vp1}
\int_{\Om_\eta}\vp_1^2\leq C\|\vp_1\|_{H^{\frac12}(\partial\Om_\eta)}^2\leq C\|\vp_1\|_{H^{1}(\Om_\eta^c)}^2\leq C\int_{\Om_\eta^c}\vp_1^2\,.
\end{equation}

Now we are ready to show that $\la_{2,\ep}$ is lower bounder by a positive $\ep$-independent constant. Using the orthogonality relation~\eqref{orthog} and the estimates~\eqref{eq.est} and~\eqref{vp1} we write, for
any nonzero $\vp\in W$ with $W\in \cW'$,
\begin{align*}
q_\ep(\vp)&=\frac{\int_M|d\vp_1|^2_\ep \, dV_\ep +
  \int_M|d\vp_2|^2_\ep \, dV_\ep}{\int \vp^2\, dV_\ep}\\
&\geq(1+C\ep^{\frac12})\frac{\int_M|d\vp_1|^2_\ep \, dV_\ep
  }{\int\vp_1^2\, dV_\ep}\\
&\geq (1+C\ep^{\frac12})\frac{\int_{\Om_\eta^c}|d\vp_1|^2}{\int_{\Om_\eta^c}\vp_1^2+ (\frac{\ep}{\ka})^{\frac d2}\int_{\Om_\eta}\vp_1^2}\\
&\geq (1+C\ep^{\frac12}]\frac{\int_{\Om_\eta^c}|d\vp_1|^2}{\int_{\Om_\eta^c}\vp_1^2}\\
&= (1+C\ep^{\frac12})\, q_{\Om_\eta^c}(\vp_1|_{\Om_\eta^c})\,.
\end{align*}
Here
\[
q_{\Om_\eta^c}(\psi):=\frac{\int_{\Om_\eta^c}|d\psi|^2}{\int_{\Om_\eta^c} \psi^2}
\]
is the Rayleigh quotient in $\Om_\eta^c$. From min-max
formulation~\eqref{minmax2} we then infer
\begin{align*}
\la_{2,\ep}&\geq (1+C\ep^{\frac12})\inf_{W\in\cW'} \max_{\vp\in
  W\minus\{0\}} q_{\Om_\eta^c}(\vp_1|_{\Om_\eta^c})\\
&\geq (1+C\ep^{\frac12})\, \mu_2\,,
\end{align*} 
where $\mu_2$ is the third Neumann eigenvalue of the domain $\Om_\eta^c$. Since $\Om_\eta^c$ has two connected components $\Om_\eta^{\pm}$, a simple computation shows that the first Neumann eigenvalues $\mu_0=\mu_1=0$, and $\mu_2$ is the smallest of the first nontrivial Neumann eigenvalues of $\Om_\eta^+$ and $\Om_\eta^-$, thus bounded by a positive $\ep$-independent constant. The lower estimate~\eqref{lajep} then follows, which establishes that $\la_\ep$ is a simple eigenvalue.

\section{Proof of Lemma~\ref{L.uniform}}\label{S:L2}
Let us begin by noticing that the function $f_\ep$ introduced in Step~1 can be written as
$f_\ep(x)=:F(\rho(x))$, where $\rho$ denotes the signed distance
to $\Si$, and 
\[
F(t):=\ka + (\ep-\ka)\, \uno_{[-\eta,\eta]}(t)\,,
\]
with $\uno_{[-\eta,\eta]}$ the indicator function of the interval. The dependence of $F$ on $\ep$ is not explicitly written for the sake of simplicity. Let
us now take a sequence of uniformly bounded smooth functions $F_n(t)$ that coincide with
$F(t)$ for $|t|\geq \eta$ and $|t|\leq \eta-\frac 1n$ and set
$f_n(x):=F_n(\rho(x))$, which is a smooth function because the signed
distance is smooth in the region where $F_n$ is nonconstant. 

It is 
standard (see e.g.~\cite{BU83}) that, as $n\to\infty$, the $k^{\text{th}}$ eigenvalue associated with the smooth metric $g_n:=f_ng_0$
converges to that of the discontinuous metric $g_\ep$ and that, if its
multiplicity is one, the corresponding eigenfunction converges
to the eigenfunction of the metric $g_\ep$ in $H^1$, possibly up to a
choice of normalizing factor. In particular, by Lemma~\ref{L} if
$u_n$ denotes the first nontrivial eigenfunction of $g_n$ with eigenvalue $\la_n$, we have that
\begin{equation}\label{H1uepun}
\lim_{n\to\infty}|\la_n-\la_\ep|=0\,, \qquad \lim_{n\to\infty}\|u_n-u_\ep\|_{H^1}=0\,,
\end{equation}
provided that $u_n$ is normalized so that
\begin{equation}\label{eq:unnorm}
\int u_n^2\, dV_n=1\,,
\end{equation}
with $dV_n$ the volume measure associated with the
metric~$g_n$. Observe that the volume of $M$ with the metric $g_n$ is
different from the volume with the metric $g_\ep$ (and hence $g_0$),
but this will not be relevant for our purposes. The
eigenfunctions~$u_n$ are smooth as the metric $g_n$ is
smooth too.

First, let us establish the energy estimate
\begin{equation}\label{unH}
\|X_1^{j_1}\cdots X_s^{j_s}u_n\|_{H^1} < C_\ep\,, 
\end{equation}
for $0\leq j_1+\cdots +j_s\leq k$, where of course the constant $C_\ep$ depends on the number of derivatives
that we consider and on $\varepsilon$, but not on $n$. For this, it is convenient to cover any given point in $U_1$ by a
patch $V$ of local coordinates $(\rho,y)$, with $y=(y_1,\dots, y_{d-1})$,
so that the metric $g_0$ reads as
\begin{equation}\label{g0ga}
g_0= d\rho^2+ \ga_{ij}(\rho,y)\, dy^i \, dy^j\,.
\end{equation}
In order to do this, let $\phi_t$ denote the local flow of the vector
field $\nabla\rho$, which is well defined in $U_0$. Any
point $x$ in $U_0$ can be written in a unique way as
\[
x=\phi_t p
\]
with $t\in I\supset [-\eta,\eta]$ and $p\in\Si$. It
we take local coordinates
$\hat y_i$ in an open set $\Si'\subset \Si$, it is standard that the desired local coordinates
$(\rho,y_i)$ on $M$ can then be defined by setting
\[
\rho(\phi_t x):= t\,,\qquad y_i(\phi_tp):= \hat y_i(p)\,.
\]
(Notice that $\rho$ is indeed the signed distance to $\Si$.)
The estimate~\eqref{unH} will clearly follow if we prove that the estimate
\begin{equation}\label{2un}
\|D_y^k u_n\|_{H^1(V)}< C_\ep
\end{equation}
holds for all $k$ (the constant $C_\ep$ is not uniform on $k$, of course).

To prove~\eqref{2un} for $k=0$, we integrate the eigenvalue equation
against a function~$\vp$ to find
\begin{equation}\label{weakn}
\int g_n^{ij}\, \pd_i u_n \, \pd_j \vp\, dV_n=\la_n\int u_n\, \vp\, dV_n\,,
\end{equation}
which is valid for all $\vp\in H^1(M)$. In particular, since the corresponding eigenvalues $\la_n$ tend to $\la_\ep$ as
$n\to\infty$, by Lemma~\ref{L} we can assume that $\la_n\leq C\ep^{\frac{d}{2}-1}$, so by
taking $\vp:= u_n$ in~\eqref{weakn} we obtain
\begin{align*}
\int f_n^{\frac  d2-1}|du_n|^2\leq C\ep^{\frac{d}{2}-1}\,.
\end{align*}
In particular, if $V$ is a patch covered by local coordinates
$(\rho,y)$ as above, an easy computation shows that
\[
\int_V \big[ (\pd_\rho u_n)^2+|D_y u_n|^2\big]\, d\rho\, dy\leq C\,,
\]
with the obvious notation, and then Eq.~\eqref{eq:unnorm} implies
\begin{equation}\label{eq:k0}
\|u_n\|_{H^1(V)}< C_\ep\,.
\end{equation}

The estimate~\eqref{2un} for $k>0$ follows by taking in Eq.~\eqref{weakn} 
\[
\vp:= [X_1^{j_1}\cdots X_s^{j_s}][X_1^{j_1}\cdots X_s^{j_s}]u_n\,,
\]
and considering all possible $j$'s with
\[
j_1+\cdots + j_s\leq k\,.
\]
In fact, in the local coordinate patch $V$, an easy partition of unity argument shows that this is equivalent to taking in Eq.~\eqref{weakn}
\begin{equation}\label{vpgen}
\vp:= [(\chi D_y)^\be][(\chi D_y)^\be] u_n\,,
\end{equation}
where $\chi$ is a cutoff function supported in~$V$
that is equal to 1 in a certain open subset of $V$, and $\be$ is a multiindex
ranging over the set $|\be|\leq k$. 

For simplicity, let us see the details for $k=1$, which implies
estimating two derivatives of $u_n$. Denoting by $O(1)$ a smooth, possibly matrix-valued,
function bounded by constants that do not depend on~$n$ and $\ep$, we start off by writing
\begin{align*}
\int g_n^{ij}\, \pd_i u_n &\,
  \pd_j\big[(\chi\,\pd_{y_l})\chi\pd_{y_l}u_n\big]\, dV_n= \\
&=-\int f_n^{\frac{d}2-1} \partial_{y_l}\big[|g_0|^{1/2}g_0^{ij}\pd_j\chi\pd_i u_n\big] \,
  (\chi\,\pd_{y_l}u_n)\,d\rho\,dy\\
& \qquad \qquad- \int (f_n)^{\frac{d}{2}-1}\, \partial_{y_l}\big[|g_0|^{1/2}g_0^{ij}\chi\partial_iu_n\big]\, \pd_j(\chi\pd_{y_l}u_n)\,d\rho\,dy\\
&= \int (f_n)^{\frac{d}{2}-1}O(1)(\partial
  u_n)(\chi\partial_{y_l}u_n)+\int f_n^{\frac{d}{2}-1}O(1)(\partial
  u_n)\partial(\chi\partial_{y_l}u_n)\\
&\qquad \qquad
+\int f_n^{\frac{d}{2}-1}O(1)(\partial_{y_l} u_n)^2-\int f_n^{\frac{d}{2}-1}g_0^{ij}\partial_i(\chi\partial_{y_l}u_n)\partial_j(\chi\partial_{y_l}u_n)\,.
\end{align*}
In the first equality we have integrated by parts and used that, since the function $f_n$ only depends on~$\rho$, we do not pick up any derivatives
of~$f_n$ with respect to $y_l$. In the second equality we integrate with the volume element $dV_0$ (so it is omitted under the integral sign) and make explicit all the terms that are of the form $O(1)$.

Furthermore, with the above choice of~$\vp$,
the RHS of Eq.~\eqref{weakn} reads:
\begin{multline*}
\la_n\int u_n (\chi\,\pd_{y_l})\chi\pd_{y_l}u_n\, dV_n
 =\la_n\int f_n^{\frac{d}{2}} O(1)\,u_n\, (\chi\pd_{y_l}u_n) -\la_n\int f_n^{\frac{d}{2}}(\chi\,\pd_{y_l}u^n)^2\,.
\end{multline*}
Therefore, using that
\begin{equation}\label{eq:triv}
\|\chi\pd_{y_l}u_n\|_{H^1}^2\leq \|u_n\|_{H^1}^2+C_\ep\int f_n^{\frac{d}{2}-1}g_0^{ij}\partial_i(\chi\partial_{y_l}u_n)\partial_j(\chi\partial_{y_l}u_n)
\end{equation}
and the Schwartz inequality, we readily arrive at:
\[
\|\chi\pd_{y_l}u_n\|_{H^1}^2\leq
C_\ep(\|u_n\|_{H^1}^2+\|u_n\|_{H^1}\|\chi\pd_{y_l}u_n\|_{H^1}) \,.
\]
Then, for $k\leq 1$ Eq.~\eqref{eq:k0} yields
\[
\|D_y^ku_n\|_{H^1(V)}< C_\ep
\]
after taking the derivatives with respect to $y_l$ for all $1\leq l\leq d-1$ in the coordinate patch $V$. Covering the set $U_0$ with coordinates
patches of the above form, we derive the upper bound
$$\|X_ju_n\|_{H^1} < C_\ep\,,$$
for all $1\leq j\leq s$. Notice that the reason for which we obtain
$\ep$-dependent constants is that $f_n$ is of order $\ep$ in
$\Om_\eta$, and it appears as a factor both in Eq.~\eqref{eq:unnorm}
and the integral of the RHS of the estimate~\eqref{eq:triv}. The case
of general $k$ follows from a completely analogous reasoning using the
function~$\vp$ specified in~\eqref{vpgen}, so we will omit the
details. The desired estimate~\eqref{unH} then follows.

Having already established~\eqref{unH}, the estimate
\begin{equation}\label{estp1}
\|u_n\|_{H^1H^k_{\mathrm T}}< C_\ep
\end{equation}
is now almost immediate. Indeed, in any coordinate patch $W\subset M$
whose closure does not intersect $\pd\Om_\eta$ and which is covered by local coordinates
$x=(x^1,\dots, x^d)$, the
components of the metric $g_n$ are bounded as $|D_x^k g_n^{ij}|\leq C_\ep$. Hence it is standard (and follows readily from the
above computation, with $\vp:=[(\chi D_x)^\be][(\chi D_x)^\be] u_n$
for some smooth cut-off function~$\chi$ supported in~$W$) that 
\[
\|D_x^k u_n\|_{L^2(W)}\leq C_\ep
\]
for any $k$. This establishes~\eqref{estp1} after covering the compact
manifold $M$ by a finite number of coordinate patches.

To prove the uniform estimate
\begin{equation}\label{estp2}
\|u_n\|_{H^1H^k_{\mathrm T}(\Om_\eta^c)}< C\,,
\end{equation}
we can use essentially the same argument. For $k=0$, we just observe that
$$\|u_n\|^2_{H^1(\Om_\eta^c)}< C$$
where we have used Eqs.~\eqref{eq:unnorm} and~\eqref{weakn}, the bound
for the eigenvalue $\la_n$ and the
fact that on $\Om_\eta^c$ we have
\begin{equation}\label{unifgn}
g_n^{ij}\, \pd_i\psi\, \pd_j\psi\geq C\big((\pd_\rho\psi)^2 + |D_y\psi|^2\big)
\end{equation}
with a constant $C>0$ that does not depend on~$\ep$.

Similarly, for
$k=1$, we use a coordinate patch $V$ covered by
coordinates $(\rho,y^1,\dots, y^{d-1})$ as above and an adapted
cutoff function $\chi$. Since we have a uniform bound
$|D_y^k g_n^{ij}|\leq C$ in $V$ by the definition of the coordinates,
we can write
\begin{equation}\label{eqrol}
\|\chi \partial_{y_l}u_n\|_{H^1(V\cap\Om_\eta^c)}^2<C\|u_n\|_{H^1(\Om_\eta^c)}^2+C\int f_n^{\frac{d}{2}-1}g_0^{ij}\partial_i(\chi \partial_{y_l}u_n)\partial_j(\chi \partial_{y_l}u_n)\,.
\end{equation}   
The preceding discussion using Eq.~\eqref{weakn} with
$\varphi=(\chi \partial_{y_l})(\chi \partial_{y_l})\,u_n$ implies that
the second term of the RHS in~\eqref{eqrol} is upper bounded by an
$\ep$-independent constant, thus establishing the uniform estimate for
$k=1$ after taking a finite covering of $M$ with $V$ patches and
summing all the contributions. For general~$k$, the reasoning is
completely analogous.  

Since the upper bounds~\eqref{estp1} and~\eqref{estp2} for $u_n$ are
uniform in $n$ and $u_n\to u_\ep$ in $H^1$, the lemma follows.

\section{Proof of Proposition~\ref{P.Ck}}\label{S:P}
A first observation is that, for any integer $k$, $u_\ep$ tends to the constants $c_\pm$ in the anisotropic Sobolev norm
$H^1H^k_{\mathrm T}(\Om_\eta^\pm)$, up to the choice of a global sign, i.e.,
\begin{equation}\label{uc}
\|u_\ep-c_\pm\|_{H^1H^k_{\mathrm T}(\Om_\eta^\pm)}=o(1)\,,
\end{equation}
where $o(1)$ stands for a quantity that tends to zero as
$\ep\searrow0$. In particular, this asymptotics implies that
\begin{equation}\label{boundary}
\lim_{\ep\searrow 0}\|B_\ep^\pm-c_\pm\|_{H^k(\partial\Om_\eta^\pm)}=0
\end{equation}
where we recall that $B_\ep^{\pm}:=u_\ep|_{\partial\Om_\eta^{\pm}}$.

In order to prove Eq.~\eqref{uc}, notice that $\la_\ep=o(1)$ by Lemma~\ref{L}, so that
\begin{align*}
\int_{\Om_\eta^c}|d u_\ep|^2&\leq\ka^{\frac d2-1}\int_{\Om_\eta^c}|d u_\ep|^2+\ep^{\frac
  d2-1}\int_{\Om_\eta}|d u_\ep|^2\\
&= \la_\ep=o(1)\,,
\end{align*}
because $\ka\geq 1$. This equation and the second estimate in Lemma~\ref{L.uniform}, which is uniform in $\ep$, implies by interpolation that there are constants $c'_\pm$ such that
$$
\|u_\ep-c'_\pm\|_{H^1H^k_{\mathrm T}(\Om_\eta^\pm)}\to 0\,,
$$
as $\ep\searrow 0$, in particular, by the Sobolev embedding theorem
for mixed norms, this implies convergence to constants in $C^\alpha(\Om_\eta^c)$ for any $\alpha<\frac12$. Let us choose the global sign of $u_\ep$ so that $\pm c'_{\pm}>0$. Finally, the fact that $c'_\pm=c_\pm$ follows from the property that $u_\ep$ is normalized and has zero mean, i.e.
\begin{align}
\ka^{\frac
    d2}\int_{\Om_\eta^c} u_\ep^2+\ep^{\frac
    d2}\int_{\Om_\eta} u_\ep^2&=1\,,\label{equ}\\
\ka^{\frac
    d2}\int_{\Om_\eta^c} u_\ep+\ep^{\frac
    d2}\int_{\Om_\eta} u_\ep&=0\,.\label{eqm}
\end{align}
Indeed, using the $L^\infty$ upper bound established in
Corollary~\ref{C:Linf}, we can take the limit $\ep\searrow 0$ in
Eqs.~\eqref{equ} and~\eqref{eqm} to obtain that, in this limit,
\begin{align*}
(c'_+)^2|\Om_\eta^+|+ (c'_-)^2|\Om_\eta^-|&=\ka^{-\frac d2}|_{\ep=0} \,,\\
c'_+|\Om_\eta^+|+ c'_-|\Om_\eta^-|&=0\,,
\end{align*}
so $c'_\pm=c_\pm$.

In the domain $\Om_\eta$ the eigenfunction $u_\ep$ satisfies the boundary problem
\begin{equation*}
\Delta u_\ep=-\ep\la_\ep u_\ep\quad\text{in } \Om_\eta\,,\qquad u_\ep|_{\partial{\Om_\eta^\pm}}=B_\ep^\pm\,,
\end{equation*}
and therefore the difference $u_\ep-h$ satisfies
\begin{equation*}
\Delta (u_\ep-h)=-\ep\la_\ep u_\ep\quad\text{in } \Om_\eta\,,\qquad (u_\ep-h)|_{\partial{\Om_\eta^\pm}}=B_\ep^\pm-c_\pm\,.
\end{equation*}
We recall that the Laplacian $\Delta$ is computed using the reference metric $g_0$. The $L^\infty$ norm of $u_\ep-h$ can be estimated using the maximum principle to yield
\begin{align*}
\|u_\ep-h\|_{L^\infty(\Om_\eta)}&\leq \|B_\ep^+-c_+\|_{L^\infty(\partial\Om_\eta^+)}+\|B_\ep^--c_-\|_{L^\infty(\partial\Om_\eta^-)}+
C \ep\la_\ep\| u_\ep\|_{L^\infty(\Om_\eta)}\\
&=o(1)\,,
\end{align*}
where we have used the asymptotics~\eqref{boundary} and
Corollary~\ref{C:Linf}. Therefore, standard elliptic
 estimates imply
\begin{align*}
\|u_\ep-h\|_{C^{k+2,\alpha}(\Om_\eta)}&<C\|u_\ep-h\|_{L^\infty(\Om_\eta)}+C\ep\la_\ep\|u_\ep\|_{C^{k,\alpha}(\Om_\eta)}\\&
<o(1)+C\ep\la_\ep\|h\|_{C^{k,\alpha}(\Om_\eta)}+C\ep\la_\ep\|u_\ep-h\|_{C^{k+2,\alpha}(\Om_\eta)}\\&
<o(1)+C\ep\la_\ep\|u_\ep-h\|_{C^{k+2,\alpha}(\Om_\eta)}\,,
\end{align*}
thus showing that
$$
\|u_\ep-h\|_{C^{k+2,\alpha}(\Om_\eta)}=o(1)\,.
$$
This completes the proof of the proposition.  

\section{Proof of Proposition~\ref{P.harm}}\label{S:prooflevel}

Let us describe the domain $\Om_\eta$ in terms of the local
coordinates $(\si,y)$, where
\[
\si:=\frac{(\rho+\eta)\pi}{2\eta}
\]
has been rescaled so that it ranges over the interval $(0,\pi)$. In
these coordinates, the Laplacian can be written as
\[
\De=:\frac{\pi^2}{4\eta^2}\,\pd_\si^2+\De_\Si-L\,,
\]
where $\De_\Si$ is the Laplacian on the hypersurface~$\Si$ (computed with
respect to the metric $\ga_{ij}(0,y)\, dy^i \, dy^j$, in the notation
of Eq.~\eqref{g0ga}) and $L$ is a differential operator of the form
\[
L=\frac{G_1}\eta\, \pd_\si + \eta G_2\, D_y^2+ \eta G_3\, D_y\,,
\]
where the functions $G_j=G_j(\si,y,\eta)$ are (possibly matrix-valued)
functions that depend smoothly on all their variables.

The difference $w:=h-\bh$ then satisfies the equation
\begin{equation}\label{eqharm}
\frac{\pi^2}{4\eta^2}\,\pd_\si^2w+\De_\Si w=Lw+\frac F\eta\,,\qquad w|_{\si=0}=w|_{\si=\pi}=0\,,
\end{equation}
with $F$ a smooth function of $(\si,y,\eta)$. In view of the Dirichlet
boundary conditions, let us expand $w$ in a Fourier series in the
variable~$\si$ of the form
\[
w=\sum_{n=1}^\infty w_n(y)\, \sin n\si\,.
\]
It is standard that the estimate presented in the statement of
Proposition~\ref{P.harm} will follow once we prove that, for any
integer $k$,
\[
\|w\|_{H^{2k}}\leq C\eta
\]
for a constant that depends on $k$, where the Sobolev norm is computed
by taking derivatives with respect to the variables $(\si,y)$ and
integrating over the set $(\si,y)\in(0,\pi)\times\Si$. 

In terms of the Fourier coefficients of $w$, this is equivalent to
showing that
\begin{equation}\label{HkFourier}
\sum_{n=1}^\infty \Big( n^{4k} \|w_n\|^2 + \|\De_\Si^kw_n\|^2\Big) <C\eta^2\,,
\end{equation}
where we are denoting by $\|\cdot\|$ the norm in $L^2(\Si)$. Let us
denote by $F_n$ and $R_n$ the $n^{\text{th}}$ Fourier coefficient of
the functions $F$ and 
\begin{equation}\label{RL}
R:=Lw\,,
\end{equation}
respectively. Writing
Eq.~\eqref{eqharm} as
\[
\frac{\pi^2n^2}{4\eta^2}w_n-\De_\Si w_n= -\frac{F_n}\eta- R_n\,,
\]
we can invert the positive self-adjoint operator
$\frac{\pi^2n^2}{4\eta^2}-\De_\Si $ in the closed manifold $\Si$ to obtain
\[
w_n=-\bigg(\frac{\pi^2n^2}{4\eta^2}-\De_\Si\bigg)^{-1}\bigg(\frac{F_n}\eta+ R_n\bigg)\,.
\]

Let us fix any integer $k\geq1$. As the $L^2(\Si)\to L^2(\Si)$ norm of the
operator $(\frac{\pi^2n^2}{4\eta^2}-\De_\Si)^{-1}$ is at most $4\eta^2/\pi^2n^2$, we then have
\begin{align*}
\sum_{n=1}^\infty n^{4k}\|w_n\|^2\leq  2 \sum_{n=1}^\infty \big( \eta^2
n^{4k-4}\|F_n\|^2 +\eta^4 n^{4k-4} \|R_n\|^2\big)\,.
\end{align*}
Notice that 
\[
\sum_{n=1}^\infty n^{4k-4}\|F_n\|^2 \leq C\|\pd_\si^{2k-2}F\|_{L^2}^2\,,
\]
where $\|\cdot\|_{L^2}$ refers to the norm computed with respect to
the variables $(\si,y)\in (0,\pi)\times\Si$ and, by the definition of
$R$ (Eq.~\eqref{RL}),
\begin{align*}
\sum_{n=1}^\infty n^{4k-4} \|R_n\|^2 \leq
C\|\pd_\si^{2k-2}R\|_{L^2}^2 \leq C\eta^{-2} \|\pd_\si^{2k-1}w\|_{L^2}^2
+C\eta^2\|\pd_\si^{2k-2}\De_\Si w\|_{L^2}^2\,.
\end{align*}
Hence
\begin{equation}\label{primera}
\sum_{n=1}^\infty n^{4k}\|w_n\|^2\leq C\eta^2\|\pd_\si^{2k-2}F\|_{L^2}^2+ C\eta^2\|w\|_{H^{2k}}^2\,.
\end{equation}

Likewise,
\begin{align}\label{segunda}
\sum_{n=1}^\infty \|\De_\Si^k w_n\|^2&\leq 
2\sum_{n=1}^\infty\bigg(\frac{\eta^2}{n^2}\|\De_\Si^k F_n \|^2 +
\bigg\|\De_\Si^k
\bigg(\frac{\pi^2n^2}{4\eta^2}-\De_\Si\bigg)^{-1}R_n\bigg\|^2\bigg)\,.
\end{align}
We can now use the definition of the function $R$ and the fact that the operator norm of
$\De_\Si(\frac{\pi^2n^2}{4\eta^2}-\De_\Si)^{-1} $ is smaller than~$1$ to write
\begin{align*}
\sum_{n=1}^\infty \bigg\|\De_\Si^k
\bigg(\frac{\pi^2n^2}{4\eta^2}-\De_\Si\bigg)^{-1}&R_n\bigg\|^2 \leq 2 \sum_{n=1}^\infty \bigg\|\De_\Si^k
\bigg(\frac{\pi^2n^2}{4\eta^2}-\De_\Si\bigg)^{-1}\bigg(\frac{G_1}\eta\, \pd_\si
w\bigg)_n\bigg\|^2 \\
& + 2\eta^2 \sum_{n=1}^\infty \bigg\|\De_\Si^k
\bigg(\frac{\pi^2n^2}{4\eta^2}-\De_\Si\bigg)^{-1}(G_2\, D_y^2w+ G_3\,
D_yw)_n\bigg\|^2\\
 &\leq 2 \sum_{n=1}^\infty \bigg\|\frac{\eta^2}{n^2}\De_\Si^k
 \bigg(\frac{G_1}\eta\, \pd_\si
 w\bigg)_n\bigg\|^2\\
&\qquad\qquad\qquad+ 2\eta^2 \sum_{n=1}^\infty \bigg\|\De_\Si^{k-1}(G_2\, D_y^2w+ G_3\,
 D_yw)_n\bigg\|^2\\
&\leq C\eta^2\|w\|_{H^{2k}}^2\,,
\end{align*}
Combining this equation with~\eqref{primera} and~\eqref{segunda} we infer that
\[
\|w\|_{H^{2k}}^2\leq C\sum_{n=1}^\infty\Big( n^{4k}\|w_n\|^2+
\|\De_\Si^k w_n\|^2\Big)\leq C\eta^2\|F\|_{H^{2k}}^2+ C\eta^2\|w\|_{H^{2k}}^2\,,
\]
which implies the estimate~\eqref{HkFourier} provided that $\eta$ is
small enough (e.g., if $C\eta^2<\frac12$). The proposition then
follows.

\section{Proof of Corollary~\ref{C.levelset}}
\label{S:coro}

Let us work with the rescaled local coordinates $(\si,y)$ introduced
in Section~\ref{S:prooflevel}. In these coordinates, the function
$\bh$ reads as
\[
\bh=c_-+\frac{c_+-c_-}\pi \si\,,
\]
so the zero set $\bh^{-1}(0)$ is $\{\pi
c_-/(c_--c_+)\}\times\Si$. Since the
derivative $\pd_\si \bh$ does not vanish and the functions $h(\si,y)$
and $\bh(\si,y)$ are arbitrarily close in $C^k((0,\pi)\times\Si)$ by
Proposition~\ref{P.harm}, Thom's isotopy theorem~\cite[Section
20.2]{AR} shows that $h^{-1}(0)$ is given by 
\[
\Psi\bigg(\bigg\{\frac{\pi
c_-}{c_--c_+}\bigg\}\times\Si\bigg)\,,
\]
where $\Psi$ is a diffeomorphism that can be taken to be arbitrarily
close to the identity in any $C^k$ norm, computed with respect to the
variables $(\si,y)$. Therefore, in the unrescaled variables
$(\rho,y)$, the diffeomorphism is $C^0$-close to the identity. (Observe
that the argument does not imply that the diffeomorphism is
$C^k$-close to the identity because the derivatives with respect to
$\rho$ introduce a large factor $\eta^{-1}$ in the estimates. In fact,
as was to be expected, what one would obtain is again some kind of
anisotropic bounds for the derivatives of $\Phi-\id$.)

\section{Proof of Theorem~\ref{T.cp}}
\label{S.cp}
Let us fix some ball $B\subset M$ and take a domain $D$ whose closure
is contained in $B$. This ensures that $\Si:=\pd D$ separates. Theorem~\ref{T.main} shows that there is a smooth metric $g$ conformal to $g_0$ and of the same volume, such that the nodal set of its first nontrivial eigenfunction $u$ is diffeomorphic to $\Si$ and the corresponding eigenvalue $\la$ is simple. Furthermore, in Step~5 of Section~\ref{S.proof} we showed that 
the gradient of $u$ does not vanish on its nodal set.

A theorem of Uhlenbeck~\cite{Uh76} ensures that one can take a metric $\tilde g$ that is a $C^{m+1}$-small conformal perturbation of the metric $g$ so that the
first eigenfunction is Morse, that is, all their critical points
are non-degenerate. It is obvious that one can take $\tilde g$ and $g$ with the same volume just multiplying by a constant factor, which does not change the eigenfunctions. Standard results from perturbation theory show
that the first nontrivial eigenfunction $\tilde u$ of the perturbed metric is close
in the $C^m(M)$ norm to $u$, so in particular the nodal set $\tilde \Si$ of
$\tilde u$ is contained in $B$ and is diffeomorphic to $\Si$. Here we are using the fact that the gradient of $u$ does not vanish on its nodal set and Thom's isotopy theorem.

Call $\tilde D$ the domain contained in $B$
that is bounded by $\tilde\Si$ and let us denote by~$\tilde\nabla$ the covariant derivative associated
with the metric $\tilde g$. Since $\tilde\nabla \tilde u$ is a
nonzero normal vector on $\tilde\Si$, which can be assumed to point
outwards without loss of generality, we can resort to Morse theory for
manifolds with boundary to show that the number of critical points of
$\tilde u$ of Morse index $i$ is at least as large as the
$i^{\text{th}}$ Betti number of the closure of the domain $\tilde D$, for $0\leq i\leq
d-1$. Since $\tilde D$ is diffeomorphic to $D$, the proposition then follows
by choosing the domain $D$ so that the sum of its Betti numbers is at
least $N$ (this can be done, e.g., by taking $\Si$ diffeomorphic
to a connected sum of $N$ copies of any nontrivial product of spheres,
such as $\SS^1\times\SS^{d-2}$, since in this case the first Betti
number is~$N$).

\section*{Acknowledgments}

The authors are supported by the ERC Starting Grants~633152 (A.E.) and~335079
(D.P.-S.). This work is supported in part by the
ICMAT--Severo Ochoa grant
SEV-2011-0087 (A.E.\ and D.P.-S.). S.S.\ was partially supported by CRC1060 of the
DFG. This work was started when S.S.\ was visiting ICMAT and he is grateful for
the enjoyable visit.

\bibliographystyle{amsplain}

\end{document}